\documentclass{article}

\usepackage{PRIMEarxiv}

\usepackage[utf8]{inputenc} 
\usepackage[T1]{fontenc}    
\usepackage{hyperref}       
\usepackage{url}            
\usepackage{booktabs}       
\usepackage{amsfonts}       
\usepackage{nicefrac}       
\usepackage{microtype}      
\usepackage{lipsum}
\usepackage{fancyhdr}       
\usepackage{graphicx}       
\graphicspath{{media/}}     
\usepackage{amsmath}
\usepackage{amsthm}

\newtheorem{theorem}{Theorem}
\newtheorem{lemma}{Lemma}

\title{Intersection of a Moran type Sierpinski carpet and a line with rational slope
\thanks{\textit{\underline{Citation}}: 
\textbf{Authors. Title. Pages.... DOI:000000/11111.}} 
}

\author{
  Simin Bao\\ 
  ChongQing University Collage of Mathematics and Statistics\\
  ChongQing\\
}

\begin{document}
\maketitle

\begin{abstract}
In 2005, Liu et al. calculated the dimensionality of the intersection of Sierpinski carpet and a straight line with rational slope in the sense of Lebesgue measure.Sierpinski carpet is a self-similar set in two-dimensional planes obtained by an iterative function system, so each layer has the same structure. While the Sierpinski carpet set with Moran structure is the limit set obtained by the action of two iterative function systems in the two-dimensional plane, which we denote as ~$F_{\sigma}$~. And the structure of each layer may be different, controlled by 0,1 sequence ~$\sigma$~ and controlled by the set. In this paper, the upper and lower box dimensions of the set ~$F_{\sigma}$~ and the straight line ~$L_{a}$~ with rational slope are calculated, where ~$a$~ is the intercept of the straight line. In addition, we consider some related problem. The main difficulty in the research is that the structure of each layer of the set ~$F_{\sigma}$~ may be different, so the structure of each layer needs to be considered with the help of the sequence ~$\sigma$~ in the calculation process.
\end{abstract}


\section{Introduction}
Fractal geometry is a new branch of mathematics founded by mathematicians ~Mandelbrot in the 20 century 70 , which provides an important mathematical method and theoretical framework for the study of irregular point sets and nonlinear phenomena. In this paper, we consider the interception problem of the fractal set, the intersection of the fractal set and the ~$(n-m)$~ dimension subspace in the Euclidean space ~$mathbb{R}^{n}$~ is called the interception of the fractal set. The study of the structure and dimension of the truncation of fractal sets is a popular topic in fractal research. ~1954~, Marstrand\cite{MR0063439} proved that the ~Hausdorff~ dimension of a horizontal truncation of a fractal set on a two-dimensional plane does not exceed the Hausdorff dimension of the fractal set minus one, in almost every sense of Lebesgue, and this conclusion is called Marstrand theorem.
\begin{theorem}[Marstrand Theorem]\cite{MR0063439}
	Let ~$A\in\mathbb{R}^{2}$~. Suppose ~$\dim_{H}(A)\ge 1$~. Let
	\begin{align*}
		A_{x}=\{y:(x,y)\in A\}
	\end{align*}
	Then for ~$Leb-a.e.$~ x,we have
	\begin{align*}
		\dim_{H}(A_{x})\le \dim_{H}(A)-1
	\end{align*}
\end{theorem}
Mattila \cite{MR0409774} gave a generalization of high-dimensional space in ~1975~. In ~2005~, Liu et al. \cite{MR2334955} studied the dimension of the intersection of Sierpinski carpet and a straight line with a slope of rational numbers, and obtained the Hausdorff dimension and box dimension of the intersection of Sierpinski carpet and a straight line with rational slope in almost every sense of the Lebesgue measure. In ~2012~ , B'{a}r'{a}ny et al.\cite{MR2929601} used the Lebesgue measure to project onto a straight line, and obtained that the dimension of the rational intercept of the Sierpinski gasket would be strictly less than the dimension of the Sierpinski gasket minus one, which was a promotion of Liu's work. In ~2014~, B'{a}r'{a}ny and Rams\cite{MR3276836} studies yielded the truncated dimension of a class of sets similar to ~Sierpinski~ carpets.

In this paper, we study a rational intercept set ~$F_{\sigma}$~ of Sierpinski carpet with Moran type, where ~$\sigma$~ is a ~0,1~ sequence that controls the structure of Sierpinski carpet set. In this paper, we calculate the dimension of the intersection of ~$F_{\sigma}$~ and the rational line ~$L_{a}:y=x\tan\theta+a$~, where ~$\theta\in[0,frac{pi}{2})$~ is one such that ~$\tan\theta$~ is the given value of the rational number; ~$a$~ is the intercept of the straight line ~$L_{a}$~, and it is easy to see ~$a\in[-\tan\theta,1]$~. Inspired by the work done by Liu et al., we calculate the upper and lower box dimensions of the rational intercept of the set ~$F_{\sigma}$~ by calculating the number of small squares intersecting ~$L_{a}$~ in each layer of the set ~$F_{\sigma}$~. In addition, this paper further considers a class of multifractal problems of rational intercepts of Sierpinski carpet sets with Moran type.

\section{The defination of Sierpinski carpet with Moran type}
\label{def}

Given two collections:
\begin{align*}
	&\Omega_{0} = \{(0,0),(0,1),(0,2),(1,0),(1,2),(2,0),(2,1),(2,2)\},\\
	&\Omega_{1} = \{(0,0),(0,1),(0,2),(0,3),(1,0),(1,3),(2,0),(2,1),(3,0),(3,1),(3,2),(3,3)\}.
\end{align*}
Define two iterative function systems as follows:
\begin{align*}
	\bigg\{ \Phi_{d}^{0}(x,y) = \frac{(x+y)+d}{3} \bigg\}_{d\in\Omega_{0}},\\
	\bigg\{ \Phi_{d}^{1}(x,y) = \frac{(x+y)+d}{4} \bigg\}_{d\in\Omega_{1}}.
\end{align*}
where~$d = (d^{1},d^{2})$~.~$\sigma = \sigma_{1}\sigma_{2}\sigma_{3}\ldots\in \{0,1\}^{\mathbb{N}}$~ is 0,1 sequence. The frequency of 0,1 in the definition sequence ~$\sigma$~ is as follows:
\begin{align*}
	n_{0}(k) &= \sharp\{1 \le i \le k :\sigma_{i} = 0\},\\
	n_{1}(k) &= \sharp\{1 \le i \le k :\sigma_{i} = 1\}.
\end{align*}
In this article, we remember
\begin{align*}
	\Phi_{d_{1}d_{2}\ldots d_{k}}^{\sigma_{1}\sigma_{2}\ldots\sigma_{k}}&=
	\Phi_{d_{1}}^{\sigma_{1}}\circ\Phi_{d_{2}}^{\sigma_{2}}\circ\ldots\circ\Phi_{d_{k}}^{\sigma_{k}},\\
	\Omega_{\sigma_{1}\sigma_{2}\ldots\sigma_{k}}&=\prod_{i=1}^{k}\Omega_{\sigma_{i}}.
\end{align*}
We define the set ~$F_{\sigma}$~ according to the following rule:
\begin{itemize}
	\item[$(1)$]  Let ~$F_{0} := [0,1] \times [0,1]$~;
	\item[$(2)$]  For any ~$j\ge1$~, if ~$\sigma_{j} = 0$~, we use the iterative function system ~${\Phi_{d}^{0}}_{d\in\Omega_{0}}$~ to act on the set ~$F_{j-1}$~ to get the set ~$F_{j}$~; If ~$\sigma_{j} = 1$~, we use the iterative function system ~${\Phi_{d}^{1}}_{d\in\Omega_{1}}$~ to act on the set ~$F_{j-1}$~ to get the set ~$F_{j}$~, then the set ~$F_{j}$~ is composed of ~$3^{n_{0}(j)}4^{n_{1}(j)}$~ small squares.
\end{itemize}
then
\begin{align*}\label{EQ21}
	F_{\sigma} = \bigcap_{j \ge 1} F_{j}.
\end{align*}
From the definition of the set ~$F_{\sigma}$~, it can be seen that the set ~$F_{\sigma}$~ has a similar structure to the common two-dimensional plane Sierpinski carpets, and also has a Moran structure, so we call such a set Sierpinski carpet with ~Moran~ type.
Define
\begin{align*}
	n_{0}=\lim_{k \to \infty}\frac{n_{0}(k)}{k},\quad
	n_{1}=\lim_{k \to \infty}\frac{n_{1}(k)}{k}
\end{align*}
then we can prove that
\begin{align*}
	\dim_{B}(F_{\sigma})=\dim_{H}(F_{\sigma})=\frac{n_{0}\log8+n_{1}\log12}{n_{0}\log3+n_{1}\log4}.
\end{align*}
If ~$\sigma=0101\ldots$~,then the image of ~$F_{\sigma}$~ of the first three steps approximation is shown in \ref{F22}.
\begin{figure*}[!ht]
	\centering
	~~~~~
	\includegraphics[width=9cm]{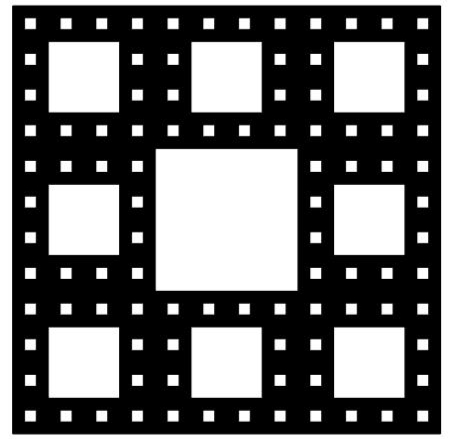}
	\caption{The first three steps of the approximation graph of ~$F_{\sigma}$~}
	\label{F22}
\end{figure*}
\vspace{-5pt}

In this paper, we will study the intersection of the set ~$F_{\sigma}$~ and a straight line with a rational slope.Let~$\theta\in[0,\frac{\pi}{2})$~,such that~$\tan\theta$~ is rational number,that is ~$\tan\theta = \frac{N}{M}$~,where~$N,M\in\mathbb{Z}$~ ,define
\begin{align*}
	L_{a}:y = x\tan\theta + a.
\end{align*}
where~$a\in\mathbb{R}$~,then ~$L_{a}$~ is a straight line with a slope of ~$\tan\theta$~ and an intercept of ~$a$~. By panning the operation we can easily find that for any given ~$\theta\in[0,\frac{\pi}{2})$~,when~$a\in[-\tan\theta,1]$~,the intersection ~$F_{\sigma} \cap L_{a} \not= \emptyset$~,We remember
\begin{align*}
	J_{\theta} := [-\tan\theta,1]=[\frac{-M}{N},1].
\end{align*}

\section{The dimension of rational intercept of Sierpinski carpet with Moran type}
\subsection{Main result}
The set ~$F_{\sigma}$~ as defined in the \ref{def} section, is the set obtained by the action of the sequence ~$\sigma$~ control iterative function system ~${\Phi_{d}^{0}}_{d\in\Omega_{0}}$~ and ~${\Phi_{d}^{1}}_{d\in\Omega_{1}}$~ on the unit square in the two-dimensional plane; ~$L_{a}$~ is a straight line which ~$\theta \in [0,\frac{\pi}{2})$~ is the inclination angle, and ~$a \in J_{\theta}$~ is the intercept, and the slope ~$\tan\theta = \frac{M}{N}$~ is a rational number, before giving the conclusion about the upper and lower box dimensions of the intersection ~$F_{\sigma} \cap L_{a}$~, let's define ~$(N+M)  \times (N+M)$~ order matrix
\begin{equation*}
	\begin{aligned}
		A_{0}^{j} &= (c_{pq}^{j})_{1 \le p,q \le N+M},\\
		c_{pq}^{j} &= \sharp\{d \in \Omega_{0} : d^{1}M - d^{2}N = 2M + 2 + q - 3p - j \}.
	\end{aligned}
\end{equation*}
where~$j \in \{0,1,2\}$~.Similarly,define ~$(N+M) \times (N+M)$~ order matrix
\begin{equation*}
	\begin{aligned}
		A_{1}^{j} &= (d_{pq}^{j})_{1 \le p,q \le N+M},\\
		d_{pq}^{j} &= \sharp\{d \in \Omega_{1} : d^{1}M - d^{2}N = 3M + 3 + q - 3p - j \}.
	\end{aligned}
\end{equation*}
where~$j \in \{0,1,2,3\}$~.

\begin{theorem}\label{T1}
	Given~$\theta \in [0,\frac{\pi}{2})$~,~$\sigma$~is a 0,1 sequence,for ~$\forall a \in J_{\theta}$~,exists ~$k \in \{1,2,\ldots,N+M\}$~,such that~$a \in [\frac{-M-1+k}{N},\frac{-M+k}{N})$~,and
	\begin{align*}
		a = \frac{-M-1+k}{N} + \frac{1}{N}\sum_{i=1}^{\infty}\frac{\xi_{i}}{3^{n_{0}(i)}4^{n_{1}(i)}}.
	\end{align*}
	where~$\xi_{i} \in \{0,1,2,3\}$~,~$(\xi_{i})$~ take the greedy exhibition,then we have
	\begin{align*}
		\overline{\dim}_{B}(F_{\sigma}\cap L_{a}) &= \limsup_{k\to\infty}
		\frac{\log \|e_{i_{0}(a)}A_{\sigma_{1}}^{\xi_{1}}A_{\sigma_{2}}^{\xi_{2}}\ldots A_{\sigma_{k}}^{\xi_{k}}\|}{n_{0}(k)\log3+n_{1}(k)\log4},\\
		\underline{\dim}_{B}(F_{\sigma} \cap L_{a}) &= \liminf_{k\to\infty}\frac{\log\|e_{i_{0}(a)}A_{\sigma_{1}}^{\xi_{1}}A_{\sigma_{2}}^{\xi_{2}}\ldots A_{\sigma_{k}}^{\xi_{k}}\|}{n_{0}(k)\log3+n_{1}(k)\log4}.
	\end{align*}
	where~$e_{i_{0}(a)}$~ is the ~$i_{0}(a)th$~ natural basis of ~$\mathbb{R}^{N+M}$~, ~$\|\|$~is the norm of the matrix,which represents the sum of all the elements in the matrix.
\end{theorem}
\subsection{The process of proof}
Given~$\theta \in [0,\frac{\pi}{2})$~,set ~$T_{d}^{0}:J_{\theta} \to J_{\theta}$~ and ~$T_{d}^{1}:J_{\theta} \to J_{\theta}$~ as follow
\begin{align*}
	\bigg\{ T_{d}^{0}(x)&=3x+d^{1}\frac{M}{N}-d^{2} \bigg\}_{d\in\Omega_{0}},\\
	\bigg\{ T_{d}^{1}(x)&=4x+d^{1}\frac{M}{N}-d^{2} \bigg\}_{d\in\Omega_{1}}.
\end{align*}
and set
\begin{align*}
	\Big\{ S_{d}^{0}(x)&=(T_{d}^{0})^{-1}:J_{\theta} \to J_{\theta} \Big\}_{d\in\Omega_{0}}\\
	\Big\{ S_{d}^{1}(x)&=(T_{d}^{1})^{-1}:J_{\theta} \to J_{\theta} \Big\}_{d\in\Omega_{1}}.
\end{align*}
~$\{S_{d}^{0}\}_{d\in\Omega_{0}}$~,~$\{S_{d}^{1}\}_{d\in\Omega_{1}}$ ~portray ~$\{S_{d}^{0}(J_{\theta})\}_{d\in\Omega_{0}}$~and~$\{S_{d}^{1}(J_{\theta})\}_{d\in\Omega_{1}}$~ correspondence with the small square of the first layer.If we project the ~$3^{n_{0} (1)}4^{n_{1}(1)}$~ small squares of the first layer along the straight line ~$L_{a}$~ onto the ~$y$~ axis,and number each small square on the first layer with ~$d$~,Let's assume~$\sigma_{1}=0$~,then~$d \in \Omega_{0}$~,we can find that ~$S_{d}^{0}(J_{\theta})$~ corresponds to a small square numbered ~$d$~;Similarly,if~$\sigma_{1}=1$~,then~$d \in \Omega_{1}$~,we also have that~$S_{d}^{1}(J_{\theta})$~ corresponds to the small square with numbered ~$d$~.As figure\ref{F31} described.
\begin{figure*}[!ht]
	\centering
	~~~~~
	\includegraphics[width=12cm]{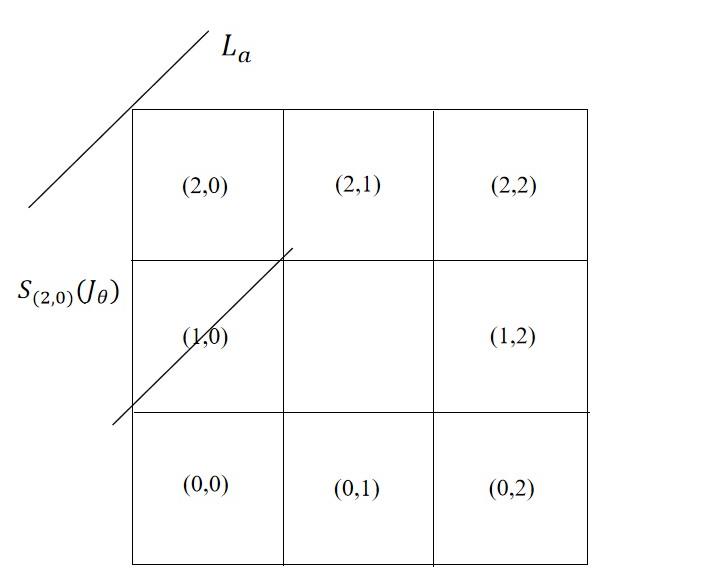}
	\caption{Mapping the function of the ~$\{S_{d}^{0}\}_{d\in\Omega_{0}}$~}
	\label{F31}
\end{figure*}
\vspace{-5pt}
We also find that
\begin{align*}
	J_{\theta}=\bigcup_{d \in \Omega_{0}}S_{d}^{0}(J_{\theta})=\bigcup_{d \in \Omega_{1}}S_{d}^{1}(J_{\theta}).
\end{align*}
We remember
\begin{align*}
	S_{d_{1}d_{2}\ldots d_{k}}^{\sigma_{1}\sigma_{2}\ldots\sigma_{k}}&=S_{d_{1}}^{\sigma_{1}}\circ S_{d_{2}}^{\sigma_{2}}\circ\ldots S_{d_{k}}^{\sigma_{k}},\\
	T_{d_{1}d_{2}\ldots d_{k}}^{\sigma_{1}\sigma_{2}\ldots\sigma_{k}}&=T_{d_{k}}^{\sigma_{k}}\circ T_{d_{k-1}}^{\sigma_{k-1}}\circ\ldots T_{d_{1}}^{\sigma_{1}}.
\end{align*}
Define
\begin{align*}
	K_{n}^{\sigma}(a)=\sharp\{d_{1}d_{2}\ldots d_{n}:\Phi_{d_{1}d_{2}\ldots d_{n}}^{\sigma_{1}\sigma_{2}\ldots\sigma_{n}}(F_{\sigma})\cap L_{a}\ne \emptyset\}
\end{align*}
By the definition of box dimension,we have
\begin{align*}
	\overline{\dim}_{B}(F_{\sigma}\cap L_{a}) &= \limsup_{k\to\infty}
	\frac{\log K_{k}^{\sigma}(a)}{n_{0}(k)\log3+n_{1}(k)\log4},\\
	\underline{\dim}_{B}(F_{\sigma} \cap L_{a}) &= \liminf_{k\to\infty}
	\frac{\log K_{k}^{\sigma}(a)}{n_{0}(k)\log3+n_{1}(k)\log4}.
\end{align*}
Define
\begin{align*}
	N_{n}^{\sigma}(a)=\sharp\{d_{1}d_{2}\ldots d_{n}:\Phi_{d_{1}d_{2}\ldots d_{n}}^{\sigma_{1}\sigma_{2}\ldots\sigma_{n}}(F_{0})\cap L_{a} \ne \emptyset\}.
\end{align*}
Since~$\Phi_{d_{1}d_{2}\ldots d_{k}}^{\sigma_{1}\sigma_{2}\ldots\sigma_{n}}(F_{0})\cap L_{a} \ne \emptyset$~ is equal to~$\Phi_{d_{1}d_{2}\ldots d_{k}}^{\sigma_{1}\sigma_{2}\ldots\sigma_{n}}(F_{\sigma})\cap L_{a}\ne \emptyset$~,then we have
\begin{align}
	\overline{\dim}_{B}(F_{\sigma}\cap L_{a}) &= \limsup_{k\to\infty}
	\frac{\log N_{k}^{\sigma}(a)}{n_{0}(k)\log3+n_{1}(k)\log4},\label{EQ319} \\
	\underline{\dim}_{B}(F_{\sigma} \cap L_{a}) &= \liminf_{k\to\infty}
	\frac{\log N_{k}^{\sigma}(a)}{n_{0}(k)\log3+n_{1}(k)\log4}.
\end{align}
The idea of our proof is to find the relationships ~$N_{n}^{\sigma}(a)$~ with ~$\|A_{\sigma_{1}}^{\xi_{1}}A_{\sigma_{2}}^{\xi_{2}}\ldots A_{\sigma_{n}}^{\xi_{n}}\|$~,The proof of the upper and lower box dimensions is consistent, so below we will only prove the upper box dimension.
\begin{lemma} \label{LL31}
	Given~$\theta\in [o,\frac{\pi}{2})$~,for all ~$a \in J_{\theta}$~,we have
	\begin{equation*}
		\begin{aligned}
			N_{n}^{\sigma}(a)&= \sharp\{d_{1}d_{2}\ldots d_{n} \in \Omega_{\sigma_{1}\sigma_{2}\ldots\sigma_{n}}:a \in S_{d_{1}d_{2}\ldots d_{n}}^{\sigma_{1}\sigma_{2}\ldots\sigma_{n}}\} \\
			&= \sharp\{d_{1}d_{2}\ldots d_{n} \in \Omega_{\sigma_{1}\sigma_{2}\ldots\sigma_{n}}:T_{d_{1}d_{2}\ldots d_{n}}^{\sigma_{1}\sigma_{2}\ldots\sigma_{n}}(a) \in J_{\theta}\}.
		\end{aligned}
	\end{equation*}
\end{lemma}

\begin{proof}
	Since the definition of ~$\{S_{d}^{0}\}_{d \in \Omega_{0}}$~ and ~$\{S_{d}^{1}\}_{d \in \Omega_{1}}$~,it is easy to find that
	\begin{align*}
		N_{n}^{\sigma}(a) &= \sharp\{d_{1}d_{2}\ldots d_{n}:\Phi_{d_{1}d_{2}\ldots d_{n}}^{\sigma_{1}\sigma_{2}\ldots\sigma_{n}}(F_{0})\cap L_{a} \ne \emptyset\} \\
		&= \sharp\{d_{1}d_{2}\ldots d_{n} \in \Omega_{\sigma_{1}\sigma_{2}\ldots\sigma_{n}}:a \in S_{d_{1}d_{2}\ldots d_{n}}^{\sigma_{1}\sigma_{2}\ldots\sigma_{n}}(J_{\theta})\} \\
		&= \sharp\{d_{1}d_{2}\ldots d_{n} \in \Omega_{\sigma_{1}\sigma_{2}\ldots\sigma_{n}}:T_{d_{1}d_{2}\ldots d_{n}}^{\sigma_{1}\sigma_{2}\ldots\sigma_{n}}(a) \in J_{\theta}\}.
	\end{align*}
	In fact,for all ~$d_{1}d_{2}\ldots d_{n}$~ such that ~$\Phi_{d_{1}d_{2}\ldots d_{n}}^{\sigma_{1}\sigma_{2}\ldots\sigma_{n}}(F_{0})\cap L_{a} \ne \emptyset$~,we remember
	\begin{align*}
		(x_{1},y_{1})=\Phi_{d_{1}d_{2}\ldots d_{k}}^{\sigma_{1}\sigma_{2}\ldots\sigma_{n}}((0,1)),\\
		(x_{2},y_{2})=\Phi_{d_{1}d_{2}\ldots d_{k}}^{\sigma_{1}\sigma_{2}\ldots\sigma_{n}}((1,0)).
	\end{align*}
	then we have
	\begin{align*}
		\Phi_{d_{1}d_{2}\ldots d_{n}}^{\sigma_{1}\sigma_{2}\ldots\sigma_{n}}(F_{0})\cap L_{a} \ne \emptyset\quad
		\Longleftrightarrow \quad a \in [y_{2}-\tan\theta x_{2},y_{1}-\tan\theta x_{1}].
	\end{align*}
	Since~$\{S_{d}^{0}(x) =(T_{d}^{0})^{-1}:J_{\theta} \to J_{\theta}\}_{d\in\Omega_{0}}$~,we can find that ~$S_{d}^{0}(x)=\frac{1}{3}x-\frac{1}{3}d^{1}\tan\theta+\frac{1}{3}d^{2},d \in \Omega_{0}$~.Similarly,~$S_{d}^{1}(x)=\frac{1}{4}x-\frac{1}{4}d^{1}\tan\theta+\frac{1}{4}d^{2},d \in \Omega_{1}$~,and then by the definition of the iterative function system~$\{\Phi_{d}^{0}\}_{d\in\Omega_{0}}$~ and ~$\{\Phi_{d}^{1}\}_{d\in\Omega_{1}}$~,we have
	\begin{align*}
		y_{2}-\tan\theta x_{2}
		&= \frac{(d_{1}^{2}-d_{1}^{1}\tan\theta)
			+\ldots+3^{n_{0}(n-1)}4^{n_{1}(n-1)}(d_{n}^{2}-d_{n}^{1}\tan\theta)}{3^{n_{0}(n)}4^{n_{1}(n)}}\\
		&=S_{d_{1}d_{2}\ldots d_{n}}^{\sigma_{1}\sigma_{2}\ldots\sigma_{n}}(-\tan\theta).
	\end{align*}
	Similarly,we can calculated that
	\begin{align*}
		y_{1}-\tan\theta x_{1} = S_{d_{1}d_{2}\ldots d_{n}}^{\sigma_{1}\sigma_{2}\ldots\sigma_{n}}(1).
	\end{align*}
	then we have
	\begin{align*}
		a \in S_{d_{1}d_{2}\ldots d_{n}}^{\sigma_{1}\sigma_{2}\ldots\sigma_{n}}(J_{\theta})\quad\Longleftrightarrow\quad\Phi_{d_{1}d_{2}\ldots d_{k}}^{\sigma_{1}\sigma_{2}\ldots\sigma_{n}}(F_{0})\cap L_{a} \ne \emptyset.
	\end{align*}
	This is where the lemma holds.
\end{proof}
Next,given ~$\theta\in[0,\frac{\pi}{2})$~,for all~$a \in J_{\theta}$~,deneted ~$\Gamma_{a}$~ as
\begin{align*}
	\Gamma_{a}=\bigg\{a+\frac{i}{N} \in J_{\theta}:i \in \mathbb{Z} \bigg\}.
\end{align*}
It is obvious that ~$\Gamma_{a}$~ has ~$N+M$~ element,that is ~$\sharp\Gamma_{a}=N+M$~. We arrange the elements of ~$\Gamma_{a}$~ as follow
\begin{align*}
	\Gamma_{a}(1)<\Gamma_{a}(2)<\ldots<\Gamma_{a}(N+M).
\end{align*}
where
\begin{align*}
	\Gamma_{a}(i) \in (\frac{-M-1+i}{N},\frac{-M+i}{N}),(1\le i \le N+M).
\end{align*}
We set
\begin{align*}
	I_{i}=(\frac{-M-1+i}{N},\frac{-M+i}{N}).
\end{align*}
Next,We would like to use a ~$(N+M) \times (N+M)$~ order matrix to record the small squares of the ~$k$~ layer,that is~$\Phi_{d_{1}d_{2}\ldots d_{k}}^{\sigma_{1}\sigma_{2}\ldots\sigma_{k}}(F_{0})$~,these small squares satisfy ~$\Phi_{d_{1}d_{2}\ldots d_{k}}^{\sigma_{1}\sigma_{2}\ldots\sigma_{k}}(F_{0}) \cap L_{a} \ne \emptyset$~. Define
\begin{equation*}
	\begin{aligned}
		B_{0}(a) &= (c_{pq})_{1 \le p,q \le N+M},\\
		c_{pq} &= \sharp\{d \in \Omega_{0} : T_{d}^{0}(\Gamma_{a}(p)) = \Gamma_{3a}(q) \}.
	\end{aligned}
\end{equation*}
And define
\begin{equation*}
	\begin{aligned}
		B_{1}(a) &= (d_{pq})_{1 \le p,q \le N+M},\\
		d_{pq} &= \sharp\{d \in \Omega_{1} : T_{d}^{1}(\Gamma_{a}(p)) = \Gamma_{4a}(q) \}.
	\end{aligned}
\end{equation*}
Then we have the lemme as follow.
\begin{lemma} \label{LL32}
	Given~$\theta\in[0,\frac{\pi}{2})$~,for all ~$a \in J_{\theta}$~,~$n \in \mathbb{N}$~,we have
	\begin{align*}
		N_{n}^{\sigma}(a)=\|e_{i_{0}(a)}B_{\sigma_{1}}(a)B_{\sigma_{2}}(3^{n_{0}(1)}4^{n_{1}(1)}a)\ldots B_{\sigma_{n}}(3^{n_{0}(n-1)}4^{n_{1}(n-1)}a)\|
	\end{align*}
	where,~$i_{0}(a) \in \{1,2,\ldots,N+M\}$~ and ~$a=\Gamma_{a}(i_{0}(a))$~.
\end{lemma}
\begin{proof}
	We use a recursive method to prove the lemma.
	\par First,if ~$b \in \Gamma_{a}$~,and ~$T_{d}^{0}(b) \in J_{\theta}$~,then there exists ~$i \in \{1,2,\ldots,N+M\}$~ such that ~$b=a+\frac{i}{N}$~,since ~$\tan\theta=\frac{M}{N}$~ is rational number,then we have
	\begin{align*}
		T_{d}^{0}(b)&=3d+d^{1}\frac{M}{N}-d^{2}\\
		&=3(a+\frac{i}{N})+d^{1}\frac{M}{N}-d^{2}\\
		&=3a+\frac{3i+d^{1}M-d^{2}N}{N}.
	\end{align*}
	Note that~$T_{d}^{0}(b) \in J_{\theta}$~,and ~$3i+d^{1}M-d^{2}N$~is an integer,then by the dedinition of ~$\Gamma_{3a}$~,we have~$T_{d}^{0}(b) \in \Gamma_{3a}$~. So we claim:if~$b \in \Gamma_{a}$~,and~$T_{d}^{0}(b) \in J_{\theta}$~,then we have~$T_{d}^{0}(b) \in \Gamma_{3a}$~;similarly,if~$b \in \Gamma_{a}$~,and~$T_{d}^{1}(b) \in J_{\theta}$~,then we have~$T_{d}^{1}(b) \in \Gamma_{4a}$~.
	\par Furthermore,if~$d \in \Omega_{0}$~ such that ~$T_{d}^{0}(\Gamma_{a}(p)) = \Gamma_{3a}(q)$~,It means:~$\Gamma_{a}(p)=S_{d}^{0}(\Gamma_{3a}(q)) \in S_{d}^{0}(J_{\theta})$~;similarly,if~$d \in \Omega_{1}$~ such that ~$T_{d}^{1}(\Gamma_{a}(p)) = \Gamma_{4a}(q)$~,it means that~$\Gamma_{a}(p) \in S_{d}^{1}(J_{\theta})$~.
	\par when ~$n=1$~,by lemma\ref{LL31},we can find that
	\begin{align*}
		N_{1}^{\sigma}(a) = \sharp\{d_{1} \in \Omega_{\sigma_{1}}:a \in S_{d_{1}}^{\sigma_{1}}(J_{\theta})\}
		= \sharp\{d_{1} \in \Omega_{\sigma_{1}}:T_{d_{1}}^{\sigma_{1}}(a) \in J_{\theta}\}.
	\end{align*}
	For all ~$d_{1} \in \{d_{1} \in \Omega_{\sigma_{1}}:T_{d_{1}}^{\sigma_{1}}(a) \in J_{\theta}\}$~,we have ~$d_{1} \in \Omega_{\sigma_{1}}$~,and ~$T_{d_{1}}^{\sigma{1}}(a) \in J_{\theta}$~.Note that~$a \in \Gamma_{a}$~,it means that there exists ~$i_{0}(a) \in \{1,2,\ldots,N+M\}$~ such that ~$a=\Gamma_{a}(i_{0}(a))$~,Let's consider ~$\sigma_{1}=0$~ and ~$\sigma_{1}=1$~.
	\par (i)When ~$\sigma_{1} = 0$~,the~$T_{d_{1}}^{0}(a) \in J_{\theta}$~,since the claim we can fine that ~$T_{d_{1}}^{0}(a) \in \Gamma_{3a}$~,it means that there exists ~$q \in \{1,2,\ldots,N+M\}$~ such that~$T_{d_{1}}^{0}(a) = \Gamma_{3a}(q)$~,then we have ~$N_{1}^{\sigma}(a) \le\sharp\{d_{1} \in \Omega_{0}:\exists q \in \{1,2,\ldots,N+M\},s.t. T_{d_{1}}^{0}(a) = \Gamma_{3a}(q)\}$~. Furthermore,if there exists~$q \in \{1,2,\ldots,N+M\}$~ such that ~$T_{d_{1}}^{0}(a) = \Gamma_{3a}(q)$~,then we have~$a \in S_{d_{1}}^{0}(J_{\theta})$~,it means that~$N_{1}^{\sigma}(a) \ge\sharp\{d_{1} \in \Omega_{0}:\exists q \in \{1,2,\ldots,N+M\},s.t. T_{d_{1}}^{0}(a) = \Gamma_{3a}(q)\}$~,then we have
	\begin{align*}
		N_{1}^{\sigma}(a) &=\sharp\{d_{1} \in \Omega_{0}:\exists q \in \{1,2,\ldots,N+M\},s.t. T_{d_{1}}^{0}(a) = \Gamma_{3a}(q)\}\\
		&=\sum_{q=1}^{N+M}\sharp\{d \in \Omega_{0}:T_{d}^{0}(a)=\Gamma_{3a}(q)\}.
	\end{align*}
	Note that
	\begin{align*}
		\|e_{i_{0}(a)}B_{0}(a)\|&=\sum_{q=1}^{N+M}c_{i_{0}q}\\
		&=\sum_{q=1}^{N+M}\sharp\{d \in \Omega_{0}:T_{d}^{0}(\Gamma_{a}(i_{0}))=\Gamma_{3a}(q)\}\\
		&=\sum_{q=1}^{N+M}\sharp\{d \in \Omega_{0}:T_{d}^{0}(a)=\Gamma_{3a}(q)\}\\
		&=N_{1}^{\sigma}(a).
	\end{align*}
	\par (ii)When ~$\sigma_{1} = 1$~,similarly
	\begin{align*}
		\|e_{i_{0}(a)}B_{0}(a)\|=N_{1}^{\sigma}(a).
	\end{align*}
	It is concluded that the lemma holds when ~$n=1$~.
	\par When~$n=2$~,by the lemma\ref{LL31} we can find that
	\begin{align*}
		N_{2}^{\sigma}(a) = \sharp\{d_{1}d_{2} \in \Omega_{\sigma_{1}\sigma_{2}}:
		a \in S_{d_{1}d_{2}}^{\sigma_{1}\sigma_{2}}(J_{\theta})\}
		= \sharp\{d_{1}d_{2} \in \Omega_{\sigma_{1}\sigma_{2}}:
		T_{d_{1}d_{2}}^{\sigma_{1}\sigma_{2}}(a) \in J_{\theta}\}.
	\end{align*}
	For all~$d_{1}d_{2} \in \{d_{1}d_{2} \in \Omega_{\sigma_{1}\sigma_{2}}:T_{d_{1}d_{2}}^{\sigma_{1}\sigma_{2}}(a) \in J_{\theta}\}$~,it is easy to find that~$d_{1} \in \Omega_{\sigma{1}}$~,~$d_{2} \in \Omega_{\sigma{2}}$~ and~$T_{d_{1}d_{2}}^{\sigma_{1}\sigma_{2}}(a) \in J_{\theta}$~,Below we will also consider the situation, at this time there are four situations, we will only elaborate on two of them, and the rest of the cases are the same.
	\par (i)When ~$\sigma_{2}=0$~ and ~$\sigma_{1}=0$~,by ~$n=1$~,we have ~$T_{d_{1}}^{0}(a)\in \Gamma_{3a}$~.By the claim we can find that ~$T_{d_{1}d_{2}}^{00}(a) \in \Gamma_{9a}$~.So there exists ~$q \in \{1,2,\ldots,N+M\}$~ such that ~$T_{d_{1}d_{2}}^{00}(a) = \Gamma_{9a}(q)$~.Furthermore,if there exists ~$q \in \{1,2,\ldots,N+M\}$~ such that ~$T_{d_{1}d_{2}}^{00}(a) = \Gamma_{9a}(q)$~,it means that~$a \in S_{d_{1}d_{2}}^{00}(J_{\theta})$~.Then we have
	\begin{align*}
		N_{2}^{\sigma}(a) &=\sharp\{d_{1}d_{2} \in \Omega_{00}:\exists q \in \{1,2,\ldots,N+M\},s.t. T_{d_{1}d_{2}}^{00}(a) = \Gamma_{9a}(q)\}\\
		&=\sum_{q=1}^{N+M}\sharp\{d_{1}d_{2} \in \Omega_{00}:T_{d_{1}d_{2}}^{00}(a) = \Gamma_{9a}(q)\}\\
		&=\sum_{q_{1}=1}^{N+M}\sum_{q_{2}=1}^{N+M}\sharp\{d_{2} \in \Omega_{0}:T_{d_{2}}^{0}(\Gamma_{3a}(q_{1}))=\Gamma_{9a}(q_{2})\}.
	\end{align*}
	Note that
	\begin{align*}
		\|e_{i_{0}(a)}B_{0}(a)B_{0}(3a)\|&=\sum_{q_{1}=1}^{N+M}c_{i_{0}q_{1}}\sum_{q_{2}=1}^{N+M}c_{q_{1}q_{2}}^{'}\\
		&=\sum_{q_{1}=1}^{N+M}\sum_{q_{2}=1}^{N+M}\sharp\{d_{2} \in \Omega_{0}:T_{d_{2}}^{0}(\Gamma_{3a}(q_{1}))=\Gamma_{9a}(q_{2})\}\\
		&=N_{2}^{\sigma}(a).
	\end{align*}
	where,we remember~$B_{0}(a)=(c_{pq})_{1 \le p,q \le N+M}$~,~$B_{0}(3a)=(c_{pq}^{'})_{1 \le p,q \le N+M}$~.
	\par (ii)when ~$\sigma_{2}=0$~ and ~$\sigma_{1}=1$~,it is similar to find that
	\begin{align*}
		N_{2}^{\sigma}(a) &=\sharp\{d_{1}d_{2} \in \Omega_{01}:\exists q \in \{1,2,\ldots,N+M\},s.t. T_{d_{1}d_{2}}^{01}(a) = \Gamma_{12a}(q)\}\\
		&=\sum_{q=1}^{N+M}\sharp\{d_{1}d_{2} \in \Omega_{01}:T_{d_{1}d_{2}}^{01}(a) = \Gamma_{12a}(q)\}\\
		&=\sum_{q_{1}=1}^{N+M}\sum_{q_{2}=1}^{N+M}\sharp\{d_{2} \in \Omega_{0}:T_{d_{2}}^{0}(\Gamma_{4a}(q_{1}))=\Gamma_{12a}(q_{2})\}.
	\end{align*}
	Note that
	\begin{align*}
		\|e_{i_{0}(a)}B_{0}(a)B_{0}(4a)\|&=\sum_{q_{1}=1}^{N+M}c_{i_{0}q_{1}}\sum_{q_{2}=1}^{N+M}c_{q_{1}q_{2}}^{'}\\
		&=\sum_{q_{1}=1}^{N+M}\sum_{q_{2}=1}^{N+M}\sharp\{d_{2} \in \Omega_{0}:T_{d_{2}}^{0}(\Gamma_{4a}(q_{1}))=\Gamma_{12a}(q_{2})\}\\
		&=N_{2}^{\sigma}(a).
	\end{align*}
	For other cases, the same can be derived
	\begin{align*}
		\|e_{i_{0}(a)}B_{0}(a)B_{0}(3a)\|=N_{2}^{\sigma}(a).
	\end{align*}
	Therefore, when ~$n=2$~, the lemma holds.
	\par And so recursively, proved that 
	\begin{align*}
		N_{n}^{\sigma}(a)=\|e_{i_{0}(a)}B_{\sigma_{1}}(a)B_{\sigma_{2}}(3^{n_{0}(1)}4^{n_{1}(1)}a)\ldots B_{\sigma_{n}}(3^{n_{0}(n-1)}4^{n_{1}(n-1)}a)\|.
	\end{align*}
	where ~$i_{0}(a) \in \{1,2,\ldots,N+M\}$~ and ~$a=\Gamma_{a}(i_{0}(a))$~. The lemma is proved.
\end{proof}

We set
\begin{align*}
	J_{p}^{j}&=\bigg(\frac{-M-1+p}{N}+\frac{j}{3N},\frac{-M-1+p}{N}+\frac{j+1}{3N}\bigg),j\in\{0,1,2\},\\
	K_{p}^{j}&=\bigg(\frac{-M-1+p}{N}+\frac{j}{4N},\frac{-M-1+p}{N}+\frac{j+1}{4N}\bigg),j\in\{0,1,2,3\}.
\end{align*}
Define 
\begin{equation*}
	\begin{aligned}
		B_{0}^{j} &= (c_{pq}^{j})_{1 \le p,q \le N+M},\\
		c_{pq}^{j} &= \sharp\{d \in \Omega_{0} : T_{d}^{0}(J_{p}^{j}) = I_{q} \}.
	\end{aligned}
\end{equation*}
And define
\begin{equation*}
	\begin{aligned}
		B_{1}^{j} &= (d_{pq}^{j})_{1 \le p,q \le N+M},\\
		d_{pq}^{j} &= \sharp\{d \in \Omega_{1} : T_{d}^{1}(K_{p}^{j}) = I_{q} \}.
	\end{aligned}
\end{equation*}
if ~$d\in\Omega_{0}$~ and such that~$T_{d}^{0}(J_{p}^{j})=I_{q}$~,by the definition ~$T_{d}^{0}$~,we can find that
\begin{equation*}
	\begin{aligned}
		T_{d}^{0}(J_{p}^{j}) &= \bigg(\frac{-3M-3+3p+j+d^{1}M-d^{2}N}{N},\frac{-3M-3+3p+j+1+d^{1}M-d^{2}N}{N} \bigg)\\
		&= \bigg(\frac{-M-1+q}{N},\frac{-M+q}{N} \bigg)\\
		&= I_{q}.
	\end{aligned}
\end{equation*}
It means that
\begin{align*}
	-3M-3+3p+j+d^{1}M-d^{2}N &= -M-1+q.
\end{align*}
From this we have
\begin{align*}
	d^{1}M-d^{2}N &= 2M+2-3p+q-j.
\end{align*}
So
\begin{equation*}
	\begin{aligned}
		c_{pq}^{j} &= \sharp\{d \in \Omega_{0} : T_{d}^{0}(J_{p}^{j}) = I_{q} \}\\
		&= \sharp\{d \in \Omega_{0} : d^{1}M - d^{2}N = 2M + 2 + q - 3p - j \}.
	\end{aligned}
\end{equation*}
It means that ~$B_{0}^{j}=A_{0}^{j}$~,where~$j \in \{0,1,2\}$~;similarly,we have~$B_{1}^{j}=A_{1}^{j}$~,where~$j \in \{0,1,2,3\}$~.
\begin{lemma}  \label{LL33}
	Given~$\theta\in[0,\frac{\pi}{2})$~,for all~$a \in J_{\theta}$~,there exists~$k \in \{1,2,\ldots,N+M\}$~,such that~$a \in [\frac{-M-1+k}{N},\frac{-M+k}{N})$~,and
	\begin{align*}
		a = \frac{-M-1+k}{N} + \frac{1}{N}\sum_{i=1}^{\infty}\frac{\xi_{i}}{3^{n_{0}(i)}4^{n_{1}(i)}}
	\end{align*}
	where~$\xi_{i} \in \{0,1,2,3\}$~,~$(\xi_{i})$~takes greedy exhibition.Furthermore,we have
	\begin{align*}\label{EQ338}
		N_{n}^{\sigma}(a)=\|e_{i_{0}(a)}A_{\sigma_{1}}^{\xi_{1}}A_{\sigma_{2}}^{\xi_{2}}\ldots A_{\sigma_{n}}^{\xi_{n}}\|.
	\end{align*}
	where ~$i_{0}(a) \in \{1,2,\ldots,N+M\}$~ and ~$a=\Gamma_{a}(i_{0}(a))$~.
\end{lemma}

\begin{proof}
	We know that,for all~$a \in J_{\theta}$~,there exists~$k \in \{1,2,\ldots,N+M\}$~,such that~$a \in [\frac{-M-1+k}{N},\frac{-M+k}{N})$~,and
	\begin{align*}
		a = \frac{-M-1+k}{N} + \frac{1}{N}\sum_{i=1}^{\infty}\frac{\xi_{i}}{3^{n_{0}(i)}4^{n_{1}(i)}}.
	\end{align*}
	where ~$\xi_{i} \in \{0,1,2,3\}$~,~$(\xi_{i})$~ take greedy exhibition. Next,we prove that
	~$N_{n}(a)=\|e_{i_{0}(a)}A_{\sigma_{1}}^{\xi_{1}}A_{\sigma_{2}}^{\xi_{2}}\ldots A_{\sigma_{n}}^{\xi_{n}}\|$~.\\
	Note that,if~$x \equiv y \bmod{\frac{1}{N}}$~,we have~$\Gamma_{x}=\Gamma_{y}$~,and we also fine that
	\begin{align*}
		B_{0}(x)=B_{0}(y),\quad B_{1}(x)=B_{1}(y).
	\end{align*}
	Then we claim:fixed~$j \in \{0,1,2\}$~,for all~$x \in \{x \in J_{\theta}:x\bmod{\frac{1}{N}} \in (\frac{j}{3N},\frac{j+1}{3N})\}$~,~$B_{0}(x)$~is a constant matrix;fixed~$j \in \{0,1,2,3\}$~,for all~$x \in \{x \in J_{\theta}:x\bmod{\frac{1}{N}} \in (\frac{j}{4N},\frac{j+1}{4N})\}$,~$B_{1}(x)$~is a constant matrix.
	\par When~$a\bmod{\frac{1}{N}} \in (\frac{j}{3N},\frac{j+1}{3N})$~,for all ~$p \in \{1,2,\ldots,N+M\}$~,we can find that~$\Gamma_{x}(p) \in J_{p}^{j}$~.For fixed~$a$~,~$J_{p}^{j}\cap\Gamma_{a}=\{\Gamma_{a}(p)\}$~,by the definition of $\{T_{d}^{0}\}_{d\in\Omega_{0}}$~,we have
	\begin{align*}
		c_{pq}=\sharp\{d \in \Omega_{0} : T_{d}^{0}(\Gamma_{a}(p))
		= \Gamma_{3a}(q) \} = \sharp\{d \in \Omega_{0} : T_{d}^{0}(J_{p}^{j}) = I_{q} \}
		=c_{pq}^{j}.
	\end{align*}
	It means that when~$a\bmod{\frac{1}{N}} \in (\frac{j}{3N},\frac{j+1}{3N})$~,we have~$B_{0}(a)=B_{0}^{j}$~;similarly when~$a\bmod{\frac{1}{N}} \in (\frac{j}{4N},\frac{j+1}{4N})$~ ,we have ~$B_{1}(a)=B_{1}^{j}$~,where~$j\in\{0,1,2,3\}$~.
	\par By the lemma\ref{LL32}we can find that
	\begin{align*}
		N_{n}^{\sigma}(a)&=\|e_{i_{0}(a)}B_{\sigma_{1}}(a)B_{\sigma_{2}}(3^{n_{0}(1)4^{n_{1}(1)}}a)\ldots B_{\sigma_{n}}(3^{n_{0}(n-1)4^{n_{1}(n-1)}}a)\|\\
		&=\|e_{i_{0}(a)}B_{\sigma_{1}}^{\xi_{1}}B_{\sigma_{2}}^{\xi_{2}}\ldots B_{\sigma_{n}}^{\xi_{n}}\|\\
		&=\|e_{i_{0}(a)}A_{\sigma_{1}}^{\xi_{1}}A_{\sigma_{2}}^{\xi_{2}}\ldots A_{\sigma_{n}}^{\xi_{n}}\|.
	\end{align*}
	This is the end of the lemma.
\end{proof}
\begin{proof}
	[The proof of Threorem \ref{T1}]
	It follow from\eqref{EQ319}that
	\begin{align*}
		\overline{\dim}_{B}(F_{\sigma}\cap L_{a}) &= \limsup_{k\to\infty}
		\frac{\log N_{k}^{\sigma}}{n_{0}(k)\log3+n_{1}(k)\log4}.
	\end{align*}
	By the lemma\ref{LL33}we can find that
	\begin{align*}
		\overline{\dim}_{B}(F_{\sigma}\cap L_{a}) &= \limsup_{k\to\infty}
		\frac{\log \|e_{i_{0}(a)}A_{\sigma_{1}}^{\xi_{1}}A_{\sigma_{2}}^{\xi_{2}}\ldots A_{\sigma_{k}}^{\xi_{k}}\|}{n_{0}(k)\log3+n_{1}(k)\log4}.
	\end{align*}
	This is the end of the theorem
\end{proof}

\section{Multifractal analysis of rational intercepts of Sierpinski carpet with Moran type}
~2009~, Feng \cite{MR2031963,MR2506331} studied the multifractal spectrum of Lyapunov exponents of the product of positive matrices, and obtained the relationship between the multifractal spectrum and the pressure function, and for general non-negative matrices, as long as the limit matrix is irreducible, there are similar results.Inspired by this, this chapter assumes the existence of the box dimension of the set ~$F_{\sigma} cap L_{a}$~, that is
\begin{align*}
	\dim_{B}(F_{\sigma} \cap L_{a}) = \lim_{k\to\infty}\frac{\log\|A_{\sigma_{1}}^{\xi_{1}}A_{\sigma_{2}}^{\xi_{2}}\ldots A_{\sigma_{k}}^{\xi_{k}}\|}{n_{0}(k)\log3+n_{1}(k)\log4}.
\end{align*}
On this basis, we will discuss the dimension of the set as follow
\begin{align*}
	E(\alpha)=\{a \in J_{\theta}:\dim_{B}(F_{\sigma} \cap L_{a})=\alpha\}.
\end{align*}
where ~$\alpha \in \mathbb{R}$~such that~$E(\alpha) \neq \emptyset$~.
\subsection{Main result}
Before giving the main conclusions of this chapter, we need to give some definitions of symbols.First,We define the symbol space ~$\Sigma_{\sigma}$~.Let ~$\Sigma_{0}=\{0,1,2\}$~,~$\quad \Sigma_{1}=\{0,1,2,3\}$~,then the definition of ~$\Sigma_{\sigma}$~ as follow
\begin{align*}
	\Sigma_{\sigma}=\prod_{i=1}^{\infty}\Sigma_{\sigma_{i}}.
\end{align*}
And the definition of the measure on that symbol space as follow
\begin{align*}
	d(x,y)=2^{-\min\{i:x_{i} \not= y_{i},x=(x_{i})_{i=1}^{\infty},y=(y_{i})_{i=1}^{\infty} \in\Sigma_{\sigma}\}}.
\end{align*}
Les~$\Sigma_{\sigma}^{n}$~ denotes a collection of all words of ~$n$ in length,~$\Sigma_{\sigma}^{\ast}$~denotes a collection of words of any length.For all~$\omega=\omega_{1}\omega_{1}\ldots\omega_{n}\in\Sigma_{\sigma}^{n}$~,~$[\omega]$~is a cylinder set,that is
\begin{align*}
	[\omega]=\{(x_{i})_{i=1}^{\infty}:x_{i}=\omega_{i},i=1,2,\ldots,n\}.
\end{align*}
~$\sigma$~is a 0,1 sequence.~$n_{0},n_{1}$~ indicates the frequency of ~$0,1$~ in the sequence ~$\sigma$~.For all~$q \in \mathbb{R}$~,define pressure function~$P_{A}(q)$~as follow
\begin{align*}
	P_{A}(q)=\limsup_{k \to \infty}\frac{\log\sum_{x\in\Sigma_{\sigma}^{k}}\|A_{\sigma_{1}}^{x_{1}}A_{\sigma_{2}}^{x_{2}}\ldots A_{\sigma_{k}}^{x_{k}}\|}{k}.
\end{align*}

\begin{theorem}\label{T2}
	Given~$\alpha \in \mathbb{R}$~,such that
	\begin{align*}
		E(\alpha)=\{a\in J_{\theta}:\dim_{B}(F_{\sigma} \cap L_{a})=\alpha\} \neq \emptyset,
	\end{align*}
	then we have
	\begin{align*}
		\dim_{H}(E(\alpha)) \le \inf_{q \in \mathbb{R}}\Big\{-q\alpha+\frac{P_{A}(q)}{n_{0}\log3+n_{1}\log4}\Big\}.
	\end{align*}
\end{theorem}

\begin{theorem}\label{T3}
	Let~$s=\dim_{H}(F_{\sigma})$~,then we have
	\begin{align*}
		\dim_{H}(E(s-1)) < 1.
	\end{align*}
\end{theorem}

\subsection{The proof of Theorem \ref{T2}}
In order to prove the main conclusion of this chapter, we need to first explain some symbols.
\par For all~$k\in\{1,2,\ldots,N+M\}$~,we define ~$\Psi_{k}:\Sigma_{\sigma}\to[\frac{-M-1+k}{N},\frac{-M+k}{N}]$~ as
\begin{align*}
	\Psi_{k}(\xi)=\frac{-M-1+k}{N}+\frac{1}{N}\sum_{i=1}^{\infty}\frac{\xi_{i}}{3^{n_{0}(i)}4^{n_{1}(i)}}.
\end{align*}
where~$\xi=(\xi_{i})_{i=1}^{\infty} \in \Sigma_{\sigma}$~.It is obvious that ~$\Psi_{k}$~is one-to-one mapping,and is bi-Lipschitz.We remember
\begin{align*}
	\lambda_{A}(x)=\lim_{k\to\infty}\frac{\log\|A_{\sigma_{1}}^{x_{1}}A_{\sigma_{2}}^{x_{2}}\ldots A_{\sigma_{k}}^{x_{k}}\|}{k}.
\end{align*}
Note that
\begin{equation} \label{EQ410}
	\begin{aligned}
		\dim_{H}(E(\alpha))&=\dim_{H}\{a\in J_{\theta}:dim_{B}(F_{\sigma} \cap L_{a})=\alpha\}\\
		&=\dim_{H}\bigg\{a\in J_{\theta}:\lim_{k\to\infty}\frac{\log\|A_{\sigma_{1}}^{\xi_{1}}A_{\sigma_{2}}^{\xi_{2}}\ldots A_{\sigma_{k}}^{\xi_{k}}\|}{n_{0}(k)\log3+n_{1}(k)\log4}=\alpha\bigg\}\\
		&=\dim_{H}\bigg\{a\in J_{\theta}:\lim_{k\to\infty}\frac{\log\|A_{\sigma_{1}}^{\xi_{1}}A_{\sigma_{2}}^{\xi_{2}}\ldots A_{\sigma_{k}}^{\xi_{k}}\|}{k}=\alpha(n_{0}\log3+n_{1}\log4)\bigg\}\\
		&=\dim_{H}\bigcup_{k=1}^{N+M}\bigg\{a\in I_{k}:\lim_{k\to\infty}\frac{\log\|A_{\sigma_{1}}^{\xi_{1}}A_{\sigma_{2}}^{\xi_{2}}\ldots A_{\sigma_{k}}^{\xi_{k}}\|}{k}=\alpha(n_{0}\log3+n_{1}\log4)\bigg\}\\
		&=\dim_{H}\{\xi\in\Sigma_{\sigma}:\exists k \in\{1,2,\ldots,N+M\},s.t. \Psi_{k}(\xi)=a,\lambda_{A}(\xi)=\alpha(n_{0}\log3+n_{1}\log4)\}\\
		&\le \dim_{H}\{\xi\in\Sigma_{\sigma}:\lambda_{A}(\xi)=\alpha(n_{0}\log3+n_{1}\log4)\}.
	\end{aligned}
\end{equation}
Therefore, to discuss the dimension of ~$E(\alpha)$~, we need to discuss the dimension of the set 
\begin{align*}
	E_{A}(\alpha)=\{\xi\in\Sigma_{\sigma}:\lambda_{A}(\xi)=\alpha\}.
\end{align*}
\begin{lemma}\label{LL41}
	\begin{align*}
		\dim_{H}(E_{A}(\alpha)) \le \frac{1}{n_{0}\log3+n_{1}\log4}\inf_{q\in\mathbb{R}}\Big\{-q\alpha+P_{A}(q)\Big\}.
	\end{align*}
\end{lemma}

\begin{proof}
	For convenient,for any~$\omega=\omega_{1}\omega_{2}\ldots\omega_{n} \in \Sigma_{\sigma}^{n}$~,and~$k \in \mathbb{N}$~,we set
	\begin{align*}
		A_{\omega}=A_{\sigma_{1}}^{\omega_{1}}A_{\sigma_{2}}^{\omega_{2}}\ldots A_{\sigma_{n}}^{\omega_{n}}.
	\end{align*}
	For any~$\epsilon >0$~,let
	\begin{align*}
		G(\alpha;k,\epsilon)=\bigcup_{n=k}^{\infty}\bigg\{x\in\Sigma_{\sigma}:\Big|\frac{1}{n}\log\|A_{x_{1}x_{2}\ldots x_{n}}\|-\alpha\Big| \le \epsilon\bigg\}.
	\end{align*}
	Then we can find that
	\begin{align*}
		E_{A}(\alpha)\subset\bigcup_{k=1}^{\infty}G(\alpha;k,\epsilon).
	\end{align*}
	By the property of Hausdorff dimension,we have
	\begin{align*}
		\dim_{H}E_{A}(\alpha)\le\dim_{H}\bigcup_{k=1}^{\infty}G(\alpha;k,\epsilon)
		\le\sup_{k}\dim_{H}G(\alpha;k,\epsilon).
	\end{align*}
	For any~$n \in \mathbb{N}$~,define
	\begin{align*}
		F(\alpha;n,\epsilon)&=\bigg\{\omega \in \Sigma_{\sigma}^{n}:\bigg|\frac{1}{n}\log\|A_{\omega}\|-\alpha \bigg| \le \epsilon\bigg\},\\
		f(\alpha;n,\epsilon)&=\sharp F(\alpha;n,\epsilon).
	\end{align*}
	Note that,For all~$n \ge k$~,we have
	\begin{align*}
		G(\alpha;k,\epsilon)\subset\bigcup_{\omega\in F(\alpha;n,\epsilon)}[\omega].
	\end{align*}
	Furthermore,for all~$\omega \in F(\alpha;n,\epsilon)$~,~$\mid[\omega]\mid=3^{-n_{0}(n)}4^{-n_{1}(n)}$~,By the property of box dimension,we have
	\begin{align*}
		\overline{\dim}_{B}G(\alpha;k,\epsilon)\le \limsup_{n\to\infty}\frac{\log(f(\alpha;n,\epsilon))}{\log3^{n_{0}(n)}4^{n_{1}(n)}}.
	\end{align*}
	Then since the relationship between the Hausdorff dimension and the box dimension,we have
	\begin{align*}
		\dim_{H}E_{A}(\alpha)&\le\sup_{k}\dim_{H}G(\alpha;k,\epsilon) \\
		&\le\sup_{k}\overline{\dim}_{B}G(\alpha;k,\epsilon)\\
		&\le \limsup_{n\to\infty}\frac{\log(f(\alpha;n,\epsilon))}{\log3^{n_{0}(n)}4^{n_{1}(n)}}.
	\end{align*}
	By the definition of ~$F(\alpha;n,\epsilon)$~,it is easy to find that for all~$\omega \in F(\alpha;n,\epsilon)$~,we have
	\begin{align*}
		\bigg|\frac{1}{n}\log\|A_{\omega}\|-\alpha\bigg| \le \epsilon
	\end{align*}
	Hence
	\begin{equation*}
		\|A_{\omega}^{q}\| \ge
		\begin{cases}
			e^{nq(\alpha-\epsilon)} & q\ge0,\\e^{nq(\alpha+\epsilon)} & q\le0.
		\end{cases}
	\end{equation*}
	Furthermore,note that
	~$\sharp F(\alpha;n,\epsilon) \le \sharp\Sigma_{\sigma}^{n}$~,then we have
	\begin{equation*}
		\sum_{\omega\in\Sigma_{\sigma}^{n}}\|A_{\omega}^{q}\| \ge
		\begin{cases}
			f(\alpha;n,\epsilon)e^{nq(\alpha-\epsilon)} & q\ge0,\\f(\alpha;n,\epsilon)e^{nq(\alpha+\epsilon)} & q\le0.
		\end{cases}
	\end{equation*}
Hence,for any~$q\in\mathbb{R}$,we have
	\begin{align*}
		P_{A}(q)&=\limsup_{k \to \infty}\frac{\log\sum_{x\in\Omega_{\omega}^{k}}\|A_{x_{1}x_{2}\ldots x_{k}}^{\omega_{1}\omega_{2}\ldots\omega_{k}}\|}{k}\\
		&\ge q\alpha+\limsup_{n\to\infty}\frac{f(\alpha;n,\epsilon)}{n}\\
		&\ge q\alpha+(n_{0}\log3+n_{1}\log4)\limsup_{n\to\infty}\frac{f(\alpha;n,\epsilon)}{\log3^{n_{0}(n)}4^{n_{1}(n)}}\\
		&\ge q\alpha+(n_{0}\log3+n_{1}\log4)\dim_{H}E_{A}(\alpha).
	\end{align*}
	Furthermore,we find that
	\begin{align*}
		\dim_{H}(E_{A}(\alpha)) \le \frac{1}{n_{0}\log3+n_{1}\log4}\inf_{q\in\mathbb{R}}\{-q\alpha+P_{A}(q)\}.
	\end{align*}
\end{proof}
\begin{proof}
	[The proof of Theorem \ref{T2}]
	It follow from\eqref{EQ410}and lemma\ref{LL41}that
	\begin{align*}
		\dim_{H}(E(\alpha))&\le\dim_{H}\{\xi\in\Sigma_{\sigma}:\lambda_{A}(\xi)=\alpha(n_{0}\log3+n_{1}\log4)\}\\
		&=\dim_{H}(E_{A}(\alpha(n_{0}\log3+n_{1}\log4))\\
		&\le\frac{1}{n_{0}\log3+n_{1}\log4}\inf_{q\in\mathbb{R}}\{-(n_{0}\log3+n_{1}\log4)q\alpha+P_{A}(q)\}\\
		&\le\inf_{q\in\mathbb{R}}\bigg\{-q\alpha+\frac{1}{n_{0}\log3+n_{1}\log4}P_{A}(q)\bigg\}.
	\end{align*}
The proof of the theorem was completed.
\end{proof}

\subsection{The proof of Theorem \ref{T3}}
Let's start by defining a measure.For any~$q \in \mathbb{R}$~,define measure~$\mu_{q}$~ on ~$\Sigma_{\sigma}$~
\begin{align*}\label{EQ414}
	\mu_{q}([x_{1}x_{2}\ldots x_{n}])&=e^{-\log\sum_{\omega \in \Sigma_{\sigma}^{n}}\| A_{\omega}\|^{q}}\| A_{x_{1}x_{2}\ldots x_{n}}\|^{q}.
\end{align*}
It is obvious that ~$\mu_{q}$~is a probability measures.

\begin{lemma} \label{LL42}
	For all~$\alpha >0$~,let~$E_{\alpha}=\bigg\{x=(x_{i})_{i=1}^{\infty}\in\Sigma_{\sigma}:\limsup_{k\to\infty}
	\frac{\|A_{\sigma_{1}}^{x_{1}}A_{\sigma_{2}}^{x_{2}}\ldots A_{\sigma_{k}\|}^{x_{k}}}{k}\ge\alpha\bigg\}$~,then we have
	\begin{align*}
		\dim_{H}(E_{\alpha})
		\le \inf_{q>0}\bigg(-q\alpha+\frac{1}{n_{0}\log3+n_{1}\log4}P_{A}(q)\bigg).
	\end{align*}
\end{lemma}

\begin{proof}
	For convenient,for any~$\omega=\omega_{1}\omega_{2}\ldots\omega_{n} \in \Sigma_{\sigma}^{n}$~,and~$k \in \mathbb{N}$~,we remember
	\begin{align*}
		A_{\omega}=A_{\sigma_{1}}^{\omega_{1}}A_{\sigma_{2}}^{\omega_{2}}\ldots A_{\sigma_{n}}^{\omega_{n}}.
	\end{align*}
	For any~$\epsilon>0$~,as well as sufficiently large~$n$~,we denote cylinder set by
	\begin{align*}
		G(\alpha;n,\epsilon)=\bigg\{[x_{1}x_{2}\ldots x_{k}]:k\ge n,\frac{1}{k}\log\| A_{x_{1}x_{2}\ldots x_{k}}\|\ge\alpha-\frac{\epsilon}{q}\bigg\}.
	\end{align*}
	It is easy to find that ~$G(\alpha;n,\epsilon)$~cover~$E_{\alpha}=\{x=(x_{i})_{i=1}^{\infty}\in\Sigma_{\sigma}:\limsup_{k \to \infty}\frac{A_{\sigma_{1}}^{x_{1}}A_{\sigma_{2}}^{x_{2}}\ldots A_{\sigma_{k}}^{x_{k}}}{k}\ge\alpha\}$~.Let~$A(\alpha;n,\epsilon)$~is the disjoint  union of ~$G(\alpha;n,\epsilon)$~and cover~$E_{\alpha}$~,then for any~$q>0$~,by the definition of  Hausdorff measure,we can find that
	\begin{align*}
		\mathcal{H}_{3^{-n_{0}(n)}4^{-n_{1}(n)}}^{-q\alpha+P_{A}(q)+2\epsilon}(E_{\alpha})\le\sum_{\omega\in A(\alpha;n,\epsilon)}(3^{-n_{0}(k)}4^{-n_{1}(k}))^{-q\alpha+P_{A}(q)+2\epsilon}.
	\end{align*}
	By the definition of ~$A(\alpha;n,\epsilon)$~and~$G(\alpha;n,\epsilon)$~,we have
 \begin{align*}
		\mathcal{H}_{3^{-n_{0}(n)}4^{-n_{1}(n)}}^{-q\alpha+P_{A}(q)+2\epsilon}(E_{\alpha})&\le \sum_{\omega\in A(\alpha;n,\epsilon)}(3^{-n_{0}(k)}4^{-n_{1}(k)})^{-q\alpha+P_{A}(q)+2\epsilon}\\
		&\le (3^{-n_{0}(n)}4^{-n_{1}(n)})^{\epsilon}\sum_{\omega\in A(\alpha;n,\epsilon)}(3^{-n_{0}(k)}4^{-n_{1}(k)})^{P_{A}(q)}(3^{-n_{0}(k)}4^{-n_{1}(k)})^{-q\alpha+\epsilon}\\
		&\le (3^{-n_{0}(n)}4^{-n_{1}(n)})^{\epsilon}\sum_{\omega\in A(\alpha;n,\epsilon)}(3^{-n_{0}(k)}4^{-n_{1}(k)})^{P_{A}(q)}(3^{-n_{0}(k)}4^{-n_{1}(k)})^{\frac{q}{k}\log\| A_{\omega}\|}\\
		&\le (3^{-n_{0}(n)}4^{-n_{1}(n)})^{\epsilon}\sum_{\omega\in A(\alpha;n,\epsilon)}(3^{-n_{0}(k)}4^{-n_{1}(k)})^{P_{A}(q)}\| A_{\omega}\|^{q}.
	\end{align*}
	Hence since the definition of ~$\mu_{q}$~,we can find that ~$C$~ is a constant,such that
	\begin{align*}
		\mathcal{H}_{3^{-n_{0}(n)}4^{-n_{1}(n)}}^{-q\alpha+P_{A}(q)+2\epsilon}(E_{\alpha})
		&\le (3^{-n_{0}(n)}4^{-n_{1}(n)})^{\epsilon}\sum_{\omega\in A(\alpha;n,\epsilon)}(3^{-n_{0}(k)}4^{-n_{1}(k)})^{P_{A}(q)}\| A_{\omega}\|^{q}\\
		&\le (3^{-n_{0}(n)}4^{-n_{1}(n)})^{\epsilon}\sum_{\omega\in A(\alpha;n,\epsilon)}C\mu_{q}([x_{1}x_{2}\ldots x_{k}])\\
		&\le C(3^{-n_{0}(n)}4^{-n_{1}(n)})^{\epsilon}.
	\end{align*}
	By the arbitrary of ~$\epsilon$~ and ~$q$~,we can find that
	\begin{align*}
		\dim_{H}(E_{\alpha})\le\inf_{q>0}\bigg(-q\alpha+\frac{1}{n_{0}\log3+n_{1}\log4}P_{A}(q)\bigg).
	\end{align*}
	This completes the proof of the lemma.
\end{proof}

\begin{proof}
	[The proof of Theorem \ref{T3}]
	It follows from\eqref{EQ410} that
	\begin{align*}
		\dim_{H}(E(s-1))&\le\dim_{H}\{\xi\in\Sigma_{\sigma}:\lambda(\xi)=(s-1)(n_{0}\log3+n_{1}\log4)\}\\
		&\le\dim_{H}\{\xi\in\Sigma_{\sigma}:\lambda(\xi)\ge(s-1)(n_{0}\log3+n_{1}\log4)\}.
	\end{align*}
	the by the lemma\ref{LL42},we can find that
	\begin{align*}
		\dim_{H}(E(s-1))\le \inf_{q>0}\{-q(s-1)+P_{A}(q)\}.
	\end{align*}
	Hence,since the definition of~$P_{A}(q)$~,we have
	\begin{align*}
		P_{A}(0)=1,\quad P_{A}(1)=s.
	\end{align*}
	Furthermore,~$P_{A}(q)$~is a concave function that increases continuously.Then there exists ~$q^{'}\in[0,1]$~ such that~$P_{A}(q^{'})<(s-1)q^{'}+1$~.Then we have
	\begin{align*}
		\dim_{H}(E(s-1))\le \inf_{q>0}\{-q(s-1)+P_{A}(q)\}\le 1.
	\end{align*}
	The proof of the theorem was completed.
\end{proof}
\label{sec:others}

\section{Conclusion}
We consider the set generated by two iterative function systems.For sets similar to the set ~$F_{\sigma}$~, such as the set ~$F_{\sigma}$~ obtained by the sequence ~$\sigma\in{0,1,2}^{\mathbb{N}}$~, or even the limit set obtained by a finite iterative function system, the box dimension can be further explored.




\end{document}